\documentclass[12pt]{amsart}
\usepackage{amssymb,amsmath}
\usepackage{mathrsfs}

\newtheorem{theorem}{Theorem}[section]
\newtheorem{lemma}[theorem]{Lemma}

\theoremstyle{definition}

\newcommand{\bN}{\mathbb{N}}

\newcommand{\bQ}{\mathbb{Q}}
\newcommand{\bR}{\mathbb{R}}

\newcommand{\et}{\quad\mbox{and}\quad}

\begin{document}

\title[Numbers with almost all convergents in a Cantor set]%
  {Construction of numbers\\ with almost all convergents in a Cantor set}

\author{Damien Roy \and Johannes Schleischitz}

\begin{abstract}
In 1984, K. Mahler
asked how well elements in the Cantor middle third set
can be approximated by rational numbers from that set,
and by rational numbers outside of that set. We consider
more general missing digit sets $C$ and construct
numbers in $C$ that
are arbitrarily well approximable
by rationals in $C$, but badly approximable
by rationals outside of $C$.
More precisely, we construct them so that all but finitely many of
their convergents lie in $C$.
\end{abstract}


\maketitle

\section{Introduction and statement of the main result}
\label{intro}

Let $b\geq 3$ be an integer and $D$ be a proper subset of
$\{0,1,\ldots,b-1\}$ with at least two elements.
We consider the Cantor set $C$ which consists of all real numbers in the
interval $[0,1]$ whose base $b$ expansion involves only the
digits from the set $D$.  This is a compact subset of $\bR$ of
measure zero. It is called the \emph{middle third Cantor set} when $b=3$
and $D=\{0,2\}$.  In 1984, K.~Mahler \cite{Ma1984} proposed a problem
about this set, which also applies to any Cantor set.  He asked how
well irrational elements of $C$ can be approximated by rational
numbers from $C$, and how well they can be approximated by rational
numbers outside of $C$.  A construction of Y.~Bugeaud \cite{Bu2008}
(see also \cite[\S 2.2]{Sc2016})
generalizing earlier work \cite{Sh1979, PS1992} of J.~Shallit and
A.~J.~van der Poorten provides an interesting answer.  For any
monotone decreasing function $\psi\colon\bN\to(0,\infty)$ with
$\lim_{q\to\infty}q^2\psi(q)=0$, defined on the set $\bN$ of positive
integers, it yields an irrational element $\xi$ of $C$ and a
constant $c=c(b)>0$ such that
\[
 \begin{array}{ll}
  |\xi-p/q| \le \psi(q)
   &\text{for infinitely many $p/q \in \bQ\cap C$,}\\[2pt]
  |\xi-p/q| \ge c\,\psi(q)
   &\text{for all $p/q\in \bQ$,}
 \end{array}
\]
with $p/q$ in reduced form.  However, because the
construction is based on the folding lemma, such a number
$\xi$ possesses many good rational approximations $p/q$ besides
those for which $|\xi-p/q| \le \psi(q)$.  As we don't
know whether they belong to $C$ or not, we lack information
about approximation to $\xi$ by rational numbers outside
of $C$.  Our main result below is more precise in this aspect
and, at the same time, answers a question of L.~Fishman and D.~Simmons
in \cite[\S 2.1]{FS2015} by providing irrational elements
of $C$ with all but finitely many convergents inside $C$
(see \cite[Chapter I]{ScLNM} for the notion of convergents of
a real number, and the theory of continued fractions).

\begin{theorem}
 \label{mainres}
Let $C$ be as above.  Then, there is a constant $c_1$,
depending only on $b$ and $D$, with the following property.
For any $\epsilon>0$ and any function
$\Psi\colon \bN\to (0,1]$, there exists $\xi\in C$
whose convergents $p/q\in\bQ$ (in reduced form) with
denominator $q\ge c_1$ all lie in $C$ and satisfy
\begin{equation}
 \label{eq:mainres}
 \min\{\Psi(q), q^{-q}\}
  > \left| \xi - \frac{p}{q} \right|
  > c_2q^{-(1+\epsilon)q}\Psi(q)
\end{equation}
for a constant $c_2=c_2(b,\epsilon)>0$.
If $D$ contains $\{0,1\}$, we may take $c_1=1$, meaning that
all convergents of $\xi$, starting with $0/1$, belong to $C$.
\end{theorem}

In particular, the numbers $\xi$ of $C$ that we construct are
Liouville numbers that are $\psi$-approximable by rational
numbers inside $C$ and badly approximable by rational numbers
outside of $C$.  Indeed, if a fraction $p/q\in \bQ\setminus C$,
in reduced form, has denominator $q\ge c_1$, then $p/q$ is not
a convergent of $\xi$ and so $|\xi-p/q|\ge 1/(2q^2)$.  On the
other hand, a result of L.~Fishman and D.~Simmons
\cite[Corollary 1.2]{FS2015} shows the existence of a
constant $c_3=c_3(b)>0$ such that the inequality
$|\xi-p/q|\le c_3/q^2$ has infinitely many solutions
$p/q \in \bQ\setminus C$.  Thus the approximation to $\xi$
by rational numbers outside of $C$ is under control as well.
As the proof will show, we even obtain explicit
base $b$ expansions for the convergents of $\xi$ with large enough
denominators.

Note that, for a general Cantor set $C$, there may exist no element
of $C$ with all its convergents in $C$.  For example, if $b$ is a
large Fibonacci number and if $D=\{d,d+1\}$ where $d$ is the preceding
Fibonacci number, then $C\subseteq [d/(b-1),(d+1)/(b-1)]$ and all elements
of $C$ have the same initial convergents $0/1,\,1/1,\,1/2,\,2/3,\,3/5,\dots$ none
of which belong to $C$.

The original motivation for this paper was to determine
whether or not Schmidt and Summerer parametric geometry of
numbers \cite{Ro2015, SS2013} extends without
qualitative change when restricting to points of the form
$(1,\xi_1,\dots,\xi_n)$ with $\xi_1,\dots,\xi_n\in C$ instead
of the general points with $\xi_1,\dots,\xi_n\in \bR$.
For $n=1$, the question amounts to determine whether or not
for any irrational $\xi\in\bR$ there exists $\xi'\in C$ and a constant
$c>1$ such that, for any convergent $p/q$ of $\xi$ (resp.\ $p'/q'$
of $\xi'$), there exists a convergent $p'/q'$ of $\xi'$
(resp.\ $p/q$ of $\xi$) with $q\le cq'$ and $q'\le cq$.
We do not know the answer but we observed that if the denominators
of the convergents of $\xi$ grow very fast then $\xi'$ must have
essentially all its convergents in $C$ and the search for such
numbers $\xi'$ led us to the construction that we describe below.

\section{Proof of the theorem}
\label{sec:proof}

We will assume, without loss of generality, that $D$ consists of only
two digits $d_1,d_2$ with $0\le d_1 <d_2\le b-1$.  Let $D^*$
denote the monoid of finite words on the alphabet $D$ with the
product given by concatenation, and let $|w|$ denote the length of
a word $w\in D^*$.  Then, each rational number in $C$, except possibly $1$,
has an ultimately periodic base $b$ expansion of the form
\[
 (0.v\,\overline{w})_b = \frac{(vw)_b-(v)_b}{b^m(b^N-1)}
 \quad
 \text{with}
 \quad
 m=|v|
 \et
 N=|w|>0,
\]
where $v\in D^*$ is a possibly empty pre-period, and $w\in D^*$ is a non-empty
period.   The numerator in the right hand-side of the formula is the difference
of two integers $(vw)_b$ and $(v)_b$, written in base $b$.

For each non-empty word $w\in D^*$, let $w'$ be the word obtained
from $w$ by replacing its last letter or digit by the other element
of the set $D$, so that $w$ and $w'$ differ only in their last digits.
Our construction depends uniquely on the choice of a
strictly increasing sequence of non-negative integers $(m_i)_{i\ge 1}$.
We define a word $v$ and a sequence of words $(w_i)_{i\ge 1}$ in $D^*$ by
\[
 v=d_1^{m_1},
 \quad
 w_1=d_2d_1^{m_1}
 \et
 w_{i+1}=(w_i)^{m_{i+1}}w_i'
 \quad
 \text{for each $i\ge 1$.}
\]
Then the sequence of rational numbers, in reduced form,
\[
 \frac{p_i}{q_i}=(0.v\,\overline{w_i})_b
 \quad (i\ge 1)
\]
is contained in $C$ and converges to an element $\xi$ of $C$.  We claim that,
for an appropriate choice of $(m_i)_{i\ge 1}$, they are consecutive
convergents of $\xi$.  The simplest case is when $D=\{0,1\}$.  As we will
see, we can then choose $m_1=0$ so that $v$ is the empty word and
all fractions $p_i/q_i$ have purely periodic base
$b$ expansion.  The reader who wants to concentrate on this case may
skip the technical lemma \ref{lemma:choice_m1}.  We start with a simple
computation.

\begin{lemma}
\label{lemma:diff_p/q}
For each $i\ge 1$, we have
\begin{equation}
\label{eq:lemma:diff_p/q}
 \frac{p_{i+1}}{q_{i+1}}
 = \frac{p_i}{q_i} + \frac{(-1)^{i+1}u}{b^{m_1}(b^{N_{i+1}}-1)}
\end{equation}
where $u=d_2-d_1$ and $N_{i+1}=|w_{i+1}|$.
\end{lemma}

\begin{proof}
Since $w_i$ ends in $d_1$ for odd indices $i$ and in $d_2$ for even ones, we find
\[
 (0.v\,\overline{w_{i+1}})_b
 = (0.v\,\overline{w_i})_b + (-1)^{i+1}u(0.0^{m_1}\overline{\epsilon})_b
\]
where $\epsilon$ consists of $N_{i+1}-1$ zeros followed by a one. The result follows.
\end{proof}

\begin{lemma}
\label{lemma:criterion}
Suppose that the sequence $(q_i)_{i\ge 1}$ is strictly increasing.
Then  $(p_i/q_i)_{i\ge 1}$ consists of all convergents to $\xi$ with denominator
at least $q_1$ if and only if, for each $i\ge 1$, we have
$b^{m_1}(b^{N_{i+1}}-1)=uq_iq_{i+1}$.
\end{lemma}

\begin{proof}
The formula \eqref{eq:lemma:diff_p/q} from
Lemma \ref{lemma:diff_p/q} can be rewritten as
\[
 \det\begin{pmatrix} p_{i+1} &q_{i+1}\\ p_i &q_i \end{pmatrix}
 = (-1)^{i+1}\frac{uq_iq_{i+1}}{b^{m_1}(b^{N_{i+1}}-1)}
 \quad (i\ge 1).
\]
If $p_i/q_i$ and $p_{i+1}/q_{i+1}$ are consecutive convergents of $\xi$
then the above determinant is $\pm 1$ and so $uq_iq_{i+1}=b^{m_1}(b^{N_{i+1}}-1)$.
Conversely, suppose that the latter equality holds for each $i\ge 1$.
Then we have
\begin{equation*}
\label{eq1:lemma:criterion}
 \det\begin{pmatrix} p_{i+1} &q_{i+1}\\ p_i &q_i \end{pmatrix}
 = (-1)^{i+1}
 \quad (i\ge 1),
\end{equation*}
and, since $(q_i)_{i\ge 1}$ is increasing, we conclude that
the sequence $(p_i/q_i)_{i\ge 1}$ consists of all convergents of its
limit $\xi$, with denominator at least $q_1$.  We leave the verification
of this fact as an interesting exercise about continued fractions
(we do not have a precise reference to propose).
\end{proof}

The choice of $m_1$ is the most delicate part of the argument.
It depends on the factorisation of $u=d_2-d_1$ in the form
\[
 u = u_1u_2,
\]
where $u_1,u_2$ are positive integers with the prime factors of $u_1$
dividing $b$ and those of $u_2$ not dividing $b$.  In the
statement below, $\varphi$ denotes Euler's totient function.

\begin{lemma}
\label{lemma:choice_m1}
Suppose that $m_1=N-1$ where $N=\varphi(u_2^2(b-1)^2)$.
Then, we have $u_1|b^{m_1}$ and the reduced fraction
$p_1/q_1=(0.v\,\overline{w_1})_b$ satisfies $u_2q_1=b^{N}-1$.
\end{lemma}

\begin{proof}
Since $N\ge u_2(b-1)\ge b-1$, we have $m_1\ge b-2\ge
u_1-1\ge v_p(u_1)$ for any prime divisor $p$ of $u_1$,
where $v_p$ denotes the valuation at $p$. Since any such
prime $p$ divides $b$, it follows that $u_1|b^{m_1}$.

For the second assertion, set $S=1+b+\cdots+b^{N-1}$,
so that $b^N-1=(b-1)S$.  We find
\[
 (0.v\,\overline{w_1})_b
 = (0.d_1^{m_1}\overline{d_2d_1^{m_1}})_b
 = (0.\overline{d_1^{m_1}d_2})_b
 = \frac{(d_1^{m_1}d_2)_b}{b^N-1}
 = \frac{ d_1S+u}{b^N-1}.
\]
Thus, we simply need to show that $\gcd(d_1S+u, b^N-1)=u_2$ or,
equivalently, that
\[
 \min\{v_p(d_1S+u), v_p(b^N-1)\} = v_p(u_2)
\]
for every prime factor $p$ of $b^N-1$.  Fix such a prime number $p$.

Since $b$ is coprime to $u_2(b-1)$, we have $b^N\equiv 1 \mod u_2^2(b-1)^2$.
Thus, $u_2^2(b-1)$ divides $S$.  If $p$ divides $u_2(b-1)$, this implies
that
\[
 v_p(S)\ge v_p(u_2^2(b-1))>v_p(u_2)=v_p(u),
\]
so $v_p(d_1S+u)=v_p(u_2)<v_p(S)\le v_p(b^N-1)$ and we are done.
Otherwise, $p$ divides $S$ but not $u$, so it does not divide $d_1S+u$
and we are done again.
\end{proof}

\begin{lemma}
\label{lemma:recursion}
Let $u_1$ and $u_2$ be as above.  Suppose that
\begin{itemize}
  \item[(i)]
  $u_1|b^{m_1}$ \et $u_2q_1=b^{m_1+1}-1$,
  \smallskip
  \item[(ii)]
  $q_0p_1(m_2+1)\equiv -1 \mod q_1$ \quad \text{where}\quad $q_0=b^{m_1}/u_1$,\
  \smallskip
  \item[(iii)]
  $q_i|m_{i+1}$ \quad for each $i\ge 2$.
\end{itemize}
Then, for each $i\ge 1$, we have
\begin{equation}
 \label{eq:lemma:recursion}
  b^{m_1}(b^{N_i}-1)=uq_{i-1}q_i
  \quad
  \text{where}
  \quad
  N_i=|w_i|.
\end{equation}

\end{lemma}

\begin{proof}
We proceed by induction on $i$.  If $i=1$, we find
$uq_0q_1=b^{m_1}u_2q_1=b^{m_1}(b^{N_1}-1)$ since
$N_1=m_1+1$.  Suppose now that \eqref{eq:lemma:recursion}
holds for some integer $i\ge 1$.  Since $N_{i+1}=(m_{i+1}+1)N_i$,
we have
\[
 b^{N_{i+1}}-1 = (b^{N_i}-1)S_{i+1}
 \quad\text{where}\quad
 S_{i+1}=1+b^{N_i}+\cdots+b^{m_{i+1}N_i},
\]
and so Lemma \ref{lemma:diff_p/q} yields
\[
 \frac{p_{i+1}}{q_{i+1}}
 = \frac{p_i}{q_i}
   + \frac{(-1)^{i+1}u}{uq_{i-1}q_iS_{i+1}}
 = \frac{R_{i+1}}{q_{i-1}q_iS_{i+1}}
 \quad\text{where}\quad
 R_{i+1}=q_{i-1}p_iS_{i+1}+(-1)^{i+1}.
\]
To complete the induction step, we simply need to show
that $q_i$ divides $R_{i+1}$ because, since
$q_{i-1}S_{i+1}$ is coprime to $R_{i+1}$, this implies
that $q_{i+1}=q_{i-1}S_{i+1}$ and so
\[
 b^{m_1}(b^{N_{i+1}}-1)
  =b^{m_1}(b^{N_i}-1)S_{i+1}
  =uq_{i-1}q_iS_{i+1}
  =uq_iq_{i+1}.
\]
When $i=1$, we use the fact that
$b^{N_1}\equiv 1 \mod q_1$ by Condition (i).
This implies that $S_2\equiv m_2+1 \mod q_1$ and thus,
using Condition (ii), we obtain, as needed,
\[
 R_2\equiv q_0p_1(m_2+1)+1 \equiv 0 \mod q_1.
\]
Now suppose that $i>1$.  Then, \eqref{eq:lemma:recursion}
has the following two consequences.  On the one hand, in
combination with Lemma \ref{lemma:diff_p/q}, it yields
\[
 p_iq_{i-1}-q_ip_{i-1}
  =q_{i-1}q_i\left(\frac{p_i}{q_i}-\frac{p_{i-1}}{q_{i-1}}\right)
  = (-1)^{i},
\]
so $p_iq_{i-1}\equiv(-1)^{i}\mod q_i$, and thus $R_{i+1}\equiv
(-1)^i(S_{i+1}-1) \mod q_i$.  On the other hand, it shows that $q_i$
divides $b^{m_1}(b^{N_i}-1)$, so
$q_i$ divides $b^{m_1}q_i^*$ where $q_i^*=\gcd(q_i,b^{N_i}-1)$.
Modulo $q_i^*$, we have $b^{N_i}\equiv 1$ and $m_{i+1}\equiv 0$
by Condition (iii), thus $S_{i+1}\equiv m_{i+1}+1\equiv 1$.
Since $u_1|b^{m_1}$ and $N_i> N_1>m_1$, we also have
$S_{i+1}\equiv 1 \mod b^{m_1}$.  As $b^{m_1}$ and $q_i^*$ are
coprime, this implies that $b^{m_1}q_i^*$ divides $S_{i+1}-1$
and so $R_{i+1}\equiv (-1)^i(S_{i+1}-1)\equiv 0 \mod q_i$.
\end{proof}

\begin{proof}[\textbf{Proof of Theorem \ref{mainres}}]
Fix a choice of $\epsilon>0$ and of a function $\psi\colon\bN\to(0,1]$.
If $d_2\neq 1$, we take $m_1$ as in Lemma \ref{lemma:choice_m1}
so that Condition (i) of Lemma \ref{lemma:recursion} holds.
Otherwise, we have $d_1=0$ and $u=d_2=1$, and we set $m_1=0$.
This yields $p_1/q_1=(0.\overline{1})_b=1/(b-1)$, and so
Condition (i) still holds.  Moreover, in both cases, the product $q_0p_1
=p_1b^{m_1}/u_1$ is coprime to $q_1$.  Thus, the integers $m_2$
satisfying Condition (ii) of Lemma \ref{lemma:criterion} form
a congruence class modulo $q_1$.  We choose $m_2$ to be the smallest
positive element of that class with $m_2\ge q_1$ for which the
corresponding fraction $p_2/q_2=(0.v\,\overline{w_2})_b$ satisfies
$1/(q_1q_2)<\psi(q_1)$.
More generally, once $m_i$ and $p_i/q_i$ are constructed for
some index $i\ge 2$, we choose $m_{i+1}$ to be the smallest
positive multiple of $q_i$ such that $p_{i+1}/q_{i+1}
=(0.v\,\overline{w_{i+1}})_b$ satisfies
\begin{equation}
 \label{eq1:proof}
 \frac{1}{q_iq_{i+1}}<\psi(q_i).
\end{equation}
This is possible at each step $i\ge 1$ because $N_{i+1}=
|w_{i+1}|=(m_{i+1}+1)N_i$ tends to infinity with $m_{i+1}$
and so, according to Formula \eqref{eq:lemma:recursion}
in Lemma \ref{lemma:recursion}, the ratio
\begin{equation}
 \label{eq2:proof}
 \frac{1}{q_iq_{i+1}}=\frac{u}{b^{m_1}(b^{N_{i+1}}-1)}
\end{equation}
tends to $0$.

We claim that, upon putting $N_0=1$, we have
\begin{equation}
 \label{eq3:proof}
 b^{m_iN_{i-1}} \le q_i < b^{(m_i+1)N_{i-1}}=b^{N_i}
\end{equation}
for each $i\ge 1$.  For $i=1$, this follows from
\[
 b^{m_1}
   \le \frac{(b-1)b^{m_1}}{u_2}
   \le \frac{b^{m_1+1}-1}{u_2}=q_1
   \le b^{m_1+1}.
\]
If $i>1$ and if we assume that \eqref{eq3:proof} holds
for all smaller values of $i$, then we have $q_{i-1}\ge q_0$
and, since $u\le b^{m_1}\le uq_0$, we find
\[
 b^{N_i-N_{i-1}}
   \le \frac{b^{N_i}}{q_{i-1}+1}
   < \frac{b^{m_1}(b^{N_i}-1)}{uq_{i-1}}
   = q_i
   < \frac{b^{m_1+N_i}}{uq_0}
   \le b^{N_i}.
\]
So, by induction, \eqref{eq3:proof} holds for all $i\ge 1$.

In particular, the sequence $(q_i)_{i\ge 1}$ is strictly increasing
and thanks to \eqref{eq:lemma:recursion}, Lemma \ref{lemma:criterion}
shows that $(p_i/q_i)_{i\ge 1}$ is a sequence of consecutive convergents
to its limit $\xi\in C$. Fix an index $i\ge 1$.  By the theory of
continued fractions, we have
\[
 \frac{1}{2q_iq_{i+1}}
  < \left|\xi-\frac{p_i}{q_i}\right|
  < \frac{1}{q_iq_{i+1}}.
\]
According to \eqref{eq1:proof}, this implies
that $|\xi-p_i/q_i|<\psi(q_i)$.  By \eqref{eq3:proof}, we also
have $q_{i+1} \ge b^{m_{i+1}N_i}\ge q_i^{q_i}$ because $m_{i+1}\ge q_i$,
and this further yields $|\xi-p_i/q_i|\le 1/q_{i+1}<q_i^{-q_i}$.

To get a lower bound for $|\xi-p_i/q_i|$ when $i\ge 2$, we note that,
if $m_{i+1}>q_i$, then, by virtue of \eqref{eq2:proof}
and of the choice of $m_{i+1}$, we have
\[
 \psi(q_i)
 \le \frac{u}{b^{m_1}(b^{N_{i+1}-q_iN_i}-1)}
  \le \frac{1}{b^{N_{i+1}-q_iN_i}}.
\]
This is again true if $m_{i+1}=q_i$ because, by hypothesis,
$\psi(q_i)\le 1$.  Thus we obtain
\[
 \left|\xi-\frac{p_i}{q_i}\right|
  > \frac{1}{2q_iq_{i+1}}
  = \frac{u}{2b^{m_1}(b^{N_{i+1}}-1)}
  \ge \frac{1}{2b^{m_1+N_{i+1}}}
  \ge \frac{\psi(q_i)}{2b^{m_1+q_iN_i}}
  \ge \frac{\psi(q_i)}{2b^{m_1}q_i^{q_i(1+m_{i})/m_{i}}},
\]
where the last inequality uses \eqref{eq3:proof}. As $m_i$ tends to
infinity with $i$ (because $m_i\ge q_{i-1}$ for $i\ge 2$), this means
that $|\xi-p_i/q_i|>\psi(q_i)q_i^{-(1+\epsilon)q_i}$ for each sufficiently
large index $i$.

This proves the theorem with $c_1=q_1$ as $q_1$
depends only on $b$ and $u$.  If $d_2=1$, we have
$D=\{0,1\}$ and we can even take
$c_1=1$ because $0=0/1$ is the only missing
convergent of $\xi$ among $(p_i/q_i)_{i\ge 1}$, and
$0=(0.\overline{0})_b$ belongs to $C$.
\end{proof}

\section*{Acknowledgments}
The research of Johannes Schleischitz was supported by the Schr\"odinger Scholarship J 3824 of the Austrian Science Fund (FWF), while that of D.~Roy was partly supported by an NSERC discovery grant.

\smallskip
\noindent
Damien {\sc Roy} and Johannes {\sc Schleischitz}\\
Department of mathematics and statistics \\
University of Ottawa\\
585 King Edward avenue\\
Ottawa, Ontario, Canada K1N 6N5\\[5pt]
\noindent
email of D.~Roy: droy@uottawa.ca\\
email of J.~Schleischitz: johannes.schleischitz@univie.ac.at\\

\end{document}